\documentclass[12pt]{amsart}
\usepackage[utf8]{inputenc}
\usepackage{vmargin,amsmath,amsfonts,amssymb,amsthm,xcolor,amscd,latexsym,xspace,enumerate}
\usepackage[colorlinks=true,linkcolor=blue,citecolor=red]{hyperref}
\usepackage{hyperref}
\usepackage{dirtytalk}
\usepackage{leftindex}

\newtheorem{theorem}{Theorem}

\newtheorem{lemma}[theorem]{Lemma}

\newcommand{\supp}{\operatorname{supp}}

\newcommand{\Int}{\operatorname{int}}

\begin{document}

\title{On the Existence of an extremal function for the Delsarte extremal problem}
\author{Mita D. Ramabulana}

\address{Mita D. Ramabulana \endgraf
Department of Mathematics and Applied Mathematics \endgraf
University of Cape Town \endgraf
Private Bag X1, 7701, Rondebosch \endgraf
Cape Town, South Africa}
\email{rmbmit001@myuct.ac.za}

\begin{abstract} 
    In the general setting of a locally compact Abelian group $G$, the Delsarte extremal problem asks for the supremum of integrals over the collection of continuous positive definite functions $f: G \to \mathbb{R}$ satisfying $f(0) = 1$ and having $\supp f_{+} \subset \Omega$ for some measurable subset $\Omega$ of finite measure. In this paper, we consider the question of the existence of an extremal function for the Delsarte extremal problem. In particular, we show that there exists an extremal function for the Delsarte problem when $\Omega$ is closed, extending previously known existence results to a larger class of functions.
\end{abstract}

\maketitle
  
\bibliographystyle{alphabetic}

  \section*{Introduction}

The Delsarte problem arose first in the context of discrete codes. Specifically, in \cite{delsarte1}, P. Delsarte found an upper bound for the number of code words in codes satisfying some prescribed distance condition. The solution was carried out in terms of Krawtchouk polynomials and the MacWilliams transform. Subsequently, in \cite{delsarte2}, P. ~ Delsarte, J.M. Goethals, and J.J. Seidel considered another instance of the Delsarte problem in the context of spherical codes in $\mathbb{R}^d$. In particular, they derived an upper bound for certain spherical codes in terms of coefficients of Gegenbauer expansions. Various instances of the Delsarte problem occur in the context of coding theory, see for instance \cite{shtrom} and the references within. It appears in the context of sphere packing where it is used to obtain upper bounds for the density of sphere packings, see \cite{arbab1, arbab2, gorbachev, gorbachev1, pbab, koblev, kolrev, kuk, viazovska}. For more on the Delsarte problem and its generalisation to locally compact Abelian (LCA) groups, see \cite{elena-szilard}. Based on the Poisson summation formula, in \cite{cohn-elkies}, H. Cohn and N. Elkies use the Delsarte scheme to show that the sphere packing density is bounded above by the Delsarte constant computed over a class of continuous positive definite functions $f: \mathbb{R}^d \to \mathbb{R}$ satisfying Delsarte-type problem restrictions. In \cite{viazovska}, M. Viazovska showed that the bound obtained by the Delsarte scheme in $\mathbb{R}^{8}$ is exact by constructing an extremal function for the Delsarte problem in that context. So, the problem of the existence of an extremal function for the Delsarte problem seems interesting. In this direction, in \cite{marcell-zsuzsa, elena-szilard} the problem of the existence of an extremal function for a Delsarte-type extremal problem is considered for the Gorbachev class, defined as follows. Let $G$ be a LCA group. Note that we assume LCA groups to be Hausdorff. Consider

	\begin{equation*}
		 \mathcal{G}_{G}(\Omega,Q) := \bigl\{ f \in P_{1}(G) \cap L^1(G): \supp{f_{+}} \subset \Omega, \supp{\widehat{f}} \subset Q \bigr\},
	\end{equation*}
	
	where $\Omega$ is a fixed closed subset of $G$ of finite Haar measure, $Q$ is a fixed compact subset of the dual group $\widehat{G}$ of $G$, $\widehat{f}$ is the Fourier transform of $f$, $P_{1}(G)$ is the collection of continuous positive definite functions $f: G \to ~ \mathbb{R}$ satisfying $f(0)=1$, $f_{+}$ is the positive part of $f$ given by $f_{+}(g) := \max{\{f(g), 0\}}$ for all $g \in G$, and $\supp f$ is the the support of $f$, given by $ \supp f := \overline{\{g \in G: f(g) \ne 0\}}$.

In \cite{marcell-zsuzsa}, it is shown that $\mathcal{G}_{G}(\Omega,Q)$ contains an extremal function for the Delsarte constant:
	
	\begin{equation*}
		\mathcal{D}_{G}(\Omega, Q) := \sup_{f \in \mathcal{G}_{G}(\Omega,Q)} \int_{G}f(g)\mbox{d}\lambda_{G}(g),
	\end{equation*}

where $\lambda_{G}$ is a fixed Haar measure on $G$. Note that this result was also proved in \cite{elena-szilard} for $G =\mathbb{R}^d$. In this paper, we remove the condition that $\supp \widehat{f} \subset Q$. More precisely, we consider the function class
	
	\begin{equation*}
		\mathcal{G}_{G}(\Omega) := \bigl\{ f \in P_1(G) \cap L^1(G): \supp{f_{+}} \subset \Omega \bigr\},
	\end{equation*}

	where $\Omega$ is a closed subset of $G$ of finite Haar measure, and we show the existence of an extremal function for the Delsarte constant 
	
	\begin{equation*}
		\mathcal{D}_{G}(\Omega): = \sup_{f \in \mathcal{G}_{G}(\Omega)} \int_{G}f(g)\mbox{d}\lambda_{G}(g).
	\end{equation*}

It may happen that the class $\mathcal{G}_{G}(\Omega)$ is empty. This happens when $0 \notin \Int \Omega$. This is because if $f$ is positive definite, then $|f(g)| \le f(0)$ for all $g \in G$, and if $0 \notin \Int \Omega$, then $f(0) = 0$, and hence $f$ is identically 0. In which case, the class $\mathcal{G}_{G}(\Omega)$ is empty because $f(0) =1$ can not hold for a continuous positive definite function $f$ with $f_{+}$ supported in $\Omega$. On the other hand, if $\Omega$ is a neighbourhood of $0$, i.e., $0 \in \Int \Omega$, then there is a symmetric neighbourhood $U$ of $0$ such that $U + U \subset \Omega$. Then the convolution $\mathbf{1}_{U} \ast \mathbf{1}_{U}$ of the characteristic function $\mathbf{1}_{U}$ of $U$ with itself is a continuous positive definite function supported in $\Omega$. We can normalise it to obtain a function that belongs to $\mathcal{G}_{G}(\Omega)$. In other words, the class $\mathcal{G}_{G}(\Omega)$ is nonempty if and only if $\Omega$ is a neighbourhood of $0$. \\

The reader may note that the condition that $f \in L^1(G)$ is needed in general for extremal problems in various function classes, such as those considered in \cite{elena-szilard, kolrev}. However, in this case it follows from $f \in P_{1}(G)$ together with $\lambda_{G} (\Omega) < \infty$. 
\section*{Notation and Preliminaries}
	
Let $G$ be a LCA group with identity $0$. Let $C(G)$ be the space of continuous functions on $G$ to the complex numbers and put on it the topology of uniform convergence on compact sets. The dual group $\widehat{G}$ of $G$ consists of continuous homomorphisms of $G$ to the multiplicative group $\mathbb{T} = \{z \in \mathbb{C}: |z| = 1 \}$. With respect to pointwise operations and the subspace topology that it inherits as a subset of $C(G)$, $\widehat{G}$ is a LCA group. Fix a Haar measure $\lambda_{G}$ on $G$ and let $L^1(G)$ denote the usual Banach space of integrable complex-valued functions on $G$. The Fourier transform of $f \in L^1(G)$ is defined by the formula:
	
	\begin{equation*}
		\widehat{f}(\chi) := \int_{G}f(g)\overline{\chi (g)}\mbox{d}\lambda_{G}(g) \mbox{ for all } \chi \in \widehat{G}.
	\end{equation*}

The convolution $f \ast g$ of functions $f,g: G \to \mathbb{C}$ is defined as

\begin{equation*}
    f \ast g (x) := \int_{G}f(y)g(x-y)\mbox{d}\lambda_{G}(y) \mbox{ for all } x \in G,
\end{equation*}

whenever the integral exists. A function $f: G \to \mathbb{C}$ is positive definite if the inequality
		\begin{equation*}
			\sum_{i = 1}^{n}\sum_{j=1}^{n}c_{i}\overline{c_{j}}f(g_{i}-g_{j}) \ge 0 
		\end{equation*}

		holds for all choices of $n \in \mathbb{N}$, $c_{i} \in \mathbb{C}$, and $g_{i} \in G$. It is a consequence of Bochner's theorem that the Fourier transform of an integrable continuous positive definite function is non-negative. Consequently, since $0 \le \widehat{f}(0) = \int_{G}f(g)\textup{d}\lambda_{G}(g)$, it follows that the integral of an integrable continuous positive definite function is non-negative. Let $P(G) \subset C(G)$ denote the collection of continuous positive definite functions on the group $G$. Define $P_{1}(G) := \{f \in P(G): f(0)=1 \}$. Related to the notion of a positive definite function is that of an integrally positive definite function. A function $f \in L^{\infty}(G)$ is integrally positive definite if the inequality

  \begin{equation*}
      \int_{G}\int_{G}f(y-x)h(x)\overline{h(y)}\textup{d}\lambda_{G}(x)\textup{d}\lambda_{G}(y) \ge 0
  \end{equation*}

  holds for all continuous compactly supported functions $h$ on $G$. According to \cite[Theorem 1.7.3]{sasvari} an integrally positive definite function agrees locally almost everywhere with a continuous positive definite function. When $G$ is $\sigma$-compact, locally almost everywhere is the same as almost everywhere. Therefore, on a $\sigma$-compact group an integrally positive definite function agrees almost everywhere with a continuous positive definite function. The approach to the main problem of this paper is to reduce the problem to the case of $\sigma$-compact LCA groups and then extend the solution to the general case. Recall that a LCA group $G$ is $\sigma$-compact if it can be written as a countable union of compact sets.

\section*{Existence of an Extremal Function}

In this section, we show the existence of an extremal function for the Delsarte extremal problem. Our approach follows that of \cite{marcell-zsuzsa}.
If $H$ is a subgroup of a LCA group $G$ and $\varphi: H \to \mathbb{R}$ is a function, its trivial extension is the function $\widetilde{\varphi}: G \to \mathbb{R}$, given by
			
			\begin{equation*}
				\widetilde{\varphi}(g) := \begin{cases}
					\varphi(g) & \text{if } g \in H, \\
					0  & \text{if } g \in G\symbol{92}H.
				\end{cases}
			\end{equation*}

\begin{lemma}[See \cite{marcell-zsuzsa}, Lemma 6 and proof of Theorem 2]  \label{extensionlemma}
    Let $H$ be an open subgroup of a LCA group $G$, and let $\Omega$ be a subset of $H$. If a function $\varphi : H \to \mathbb{R}$ is in $\mathcal{G}_{H}(\Omega)$, then its trivial extension $\widetilde{\varphi}: G \to \mathbb{R}$ is in $\mathcal{G}_{G}(\Omega)$.
\end{lemma}

\begin{proof}
    Suppose that $\varphi: H \to \mathbb{R}$ is in $\mathcal{G}_{H}(\Omega)$. In particular, $\varphi$ is continuous, positive definite, has $\supp \varphi_{+} \subset \Omega$, and satisfies $\varphi(0) = 1$. Clearly, its trivial extension has $\supp \widetilde{\varphi}_{+} \subset \Omega$, and satisfies $\widetilde{\varphi}(0) = 1$. To see that $\widetilde{\varphi}$ is continuous, note that $H$, being an open subgroup, is also closed, and hence its complement $G \symbol {92} H$ is open. Now,  $\widetilde{\varphi}$ is continuous at all $g \in H$ because it coincides with the continuous function $\varphi: H \to \mathbb{R}$ on the open set $H$. It is continuous at all $g \in G \symbol{92} H$ because it is identically 0 on the open set $G \symbol{92} H$.  It is positive definite by  \cite[Theorem 32.43(a)]{hewittross}. 
\end{proof}

We shall also make use of the following lemma.

\begin{lemma}\label{sigmacompactgeneration}
    Let $G$ be a LCA group, and let $(K_n)_{n \in \mathbb{N}}$ be a sequence of $\sigma$-compact subsets of $G$. The group $H$ generated by the union $\bigcup_{n \in \mathbb{N}}K_n$ is $\sigma$-compact. Furthermore, if at least one of the $K_{n}$'s has non-empty interior, then $H$ is open.
\end{lemma}


Lemma \ref{sigmacompactgeneration} is standard and the interested reader may consult the proof of \cite[Proposition 1.2.1 (c)]{deitmar}.

\begin{theorem}\label{reduction}
    Let $G$ be a LCA group, and $\Omega$ a subset of $G$ of finite Haar measure. There is an open $\sigma$-compact subgroup $H$ of $G$ containing $\Omega$ such that 

    \begin{equation*}\label{subgroupinequality}
        \mathcal{D}_{G}(\Omega) = \mathcal{D}_{H}(\Omega).
    \end{equation*}
\end{theorem}

\begin{proof}
      By outer regularity of the Haar measure on $G$, there is an open identity neighbourhood $U$ in $G$ containing $\Omega$ with $\lambda_{G}(U) < \infty$. We claim that $U$ lies in an open $\sigma$-compact subgroup. To see this, let $K$ be some $\sigma$-compact open subgroup of $G$. One may always obtain such a $K$ by considering a relatively compact symmetric identity neighbourhood $V$  and observing that the union

    \begin{equation*}
        K:= \bigcup_{n = 1}^{\infty}nV
    \end{equation*}

    is an open $\sigma$-compact subgroup. Now, note that $G$ is a disjoint union of cosets $g+K$ of $K$ with $g$ running through a set of coset representatives of $K$. Since the cosets of $K$ are disjoint and cover $G$, the collection of subsets $(g+K) \cap U$ of $U$ as $g$ runs through the set of coset representatives of $K$ are disjoint and their union is $U$. Now, either $(g+K) \cap U = \emptyset$ or $\lambda_{G}((g+K) \cap U) >0$ because $(g+K) \cap U$ is open, and every open set is either empty or has strictly positive measure. Hence $(g+K) \cap U \ne \emptyset$ for at most countably many cosets $g+K$ of $K$, otherwise this would contradict the finiteness of the Haar measure of $U$. Clearly, $U$ is contained in the subgroup generated by the union of $K$ and the at most countably many open $\sigma$-compact cosets $g + K$ of $K$ that meet $U$ nontrivially. Denote this subgroup by $H$ and observe that it is open and $\sigma$-compact by Lemma \ref{sigmacompactgeneration}. Equip $H$ with Haar measure $\lambda_{H} := \lambda_{G}|_{H}$. By construction, $\Omega \subset H$. \\
    Let us first consider the case when $0 \notin \Int \Omega$. Then, both $\mathcal{G}_{G}(\Omega)$ and $\mathcal{G}_{H}(\Omega)$ are empty. Therefore, both sides of \eqref{subgroupinequality} are equal to 0, and the theorem is proved in this case. Now assume that $0 \in \Int \Omega$, so that both $\mathcal{G}_{G}(\Omega)$ and $\mathcal{G}_{H}(\Omega)$ are nonempty.\\
    First, let $(f_{n})_{n \in \mathbb{N}}$ be a $\mathcal{D}_{G}(\Omega)$-extremal sequence in $\mathcal{G}_{G}(\Omega)$ such that

    \begin{equation*}
        \int_{G}f_{n}(g)\mbox{d}\lambda_{G}(g) \ge \mathcal{D}_{G}(\Omega) - \frac{1}{n} \textup{ for all } n \in \mathbb{N}.
    \end{equation*}

Let $f_{n}|_{H}: H \to \mathbb{R}$ denote the restriction of $f_{n}$ to $H$. Observe that for all $n \in \mathbb{N}$, we have

\begin{equation}\label{splitintegral}
        \int_{G}f_{n}(g)\mbox{d}\lambda _{G}(g) = \int_{H}f_{n}(g)\mbox{d}\lambda_{G}(g) + \int_{G\symbol{92}H}f_{n}(g)\mbox{d}\lambda_{G}(g),
\end{equation}

and since $\supp (f_{n})_{+}$ is contained in $H$, we have that $f_{n}(g) \le 0$ for all $g \in G \symbol{92}H$. Therefore, 

\begin{equation}\label{negint}
    \int_{G\symbol{92}H}f_{n}(g)\mbox{d}\lambda_{G}(g) \le 0.
\end{equation}

From \eqref{splitintegral} and \eqref{negint} it follows that 

\begin{equation}\label{ineq}
    \int_{H}f_{n}|_{H}(g)\mbox{d}\lambda_{H}(g) = \int_{H}f_{n}(g)\mbox{d}\lambda_{G}(g) \ge \int_{G}f_{n}(g)\mbox{d}\lambda _{G}(g).
\end{equation}

Now, by the definition of $\mathcal{D}_{H}(\Omega)$ and inequality \eqref{ineq}, we have

\begin{equation*}
    \mathcal{D}_{H}(\Omega) \ge \int_{H}f_{n}|_{H}(g)\textup{d}\lambda_{H}(g) \ge \int_{G}f_{n}(g)\textup{d}\lambda_{G}(g) \ge \mathcal{D}_{G}(\Omega) - \frac{1}{n}.
\end{equation*}

Thus $\mathcal{D}_{H}(\Omega) \ge \mathcal{D}_{G}(\Omega)$. \\
Now, let $(h_{n})_{n \in \mathbb{N}}$ be a $\mathcal{D}_{H}(\Omega)$-extremal sequence in $\mathcal{G}_{H}(\Omega)$ such that 
			
			\begin{equation*}
				\int_{H}h_{n}(g)\mbox{d}\lambda_{H}(g) \ge \mathcal{D}_{H}(\Omega)-\frac{1}{n} \textup{ for all } n \in \mathbb{N}.
			\end{equation*}

			By Lemma \ref{extensionlemma}, the sequence $(\widetilde{h_{n}})_{n \in \mathbb{N}}$ of trivial extensions is in $\mathcal{G}_{G}(\Omega)$, and since $\widetilde{h_{n}}$ vanishes outside of $H$, we have that for all $n \in \mathbb{N}$,
			
			\begin{equation*}
				\mathcal{D}_{G}(\Omega) \ge \int_{G}\widetilde{h_{n}}(g)\textup{d}\lambda_{G}(g) = \int_{H}h_{n}(g)\textup{d}\lambda_{H}(g) \ge \mathcal{D}_{H}(\Omega) - \frac{1}{n}.
			\end{equation*}
			
			Thus, $\mathcal{D}_{G}(\Omega) \ge \mathcal{D}_{H}(\Omega)$, and hence 

   \begin{equation*}
       \mathcal{D}_{G}(\Omega) = \mathcal{D}_{H}(\Omega),
   \end{equation*}
   
   as required.

\end{proof}

Next, we show the existence of an extremal function for the Delsarte extremal problem on LCA groups. We first obtain a solution in the context of $\sigma$-compact groups and later extend it to the general case. The reader should note that the condition that $\Omega$ is closed kicks in here. We also give some information on how well the extremal function can be approximated, which needs some preliminary results.

\begin{lemma}[See \cite{sasvari}, Theorem 1.4.3]\label{integralinequality}
    Let $G$ be a LCA group and $f$ a continuous positive definite function on $G$. Let $V$ be a symmetric neighborhood of $0$ with compact closure and set $\chi:= \mathbf{1}_{V}/\lambda_{G}(V)$ where $\mathbf{1}_{V}$ is the characteristic function of $V$. Then the inequality

    \begin{equation*}\label{inequality}
        |f(x) - f \ast \chi (x)|^{2} \le 2f(0)\lambda_{G}(V)^{-1}\int_{V}(f(0)-\mathrm{Re}f(g)) \textup{d} \lambda_{G}(g)
    \end{equation*}
    holds for all $x \in G$.
\end{lemma}

In the following, we adapt \cite[Lemma 1.5.1, Theorem 1.5.2]{sasvari} to our context, replacing weak$^{*}$ convergence against $L^1(G)$ and boundedness in $L^{\infty}(G)$ by weak convergence and boundedness in $L^2(G)$ to prove essentially the same results as in \cite[Lemma 1.5.1, Theorem 1.5.2]{sasvari}. The proofs are essentially the same, but we include them here for completeness.

\begin{lemma}[Compare \cite{sasvari}, Lemma 1.5.1]\label{weakuniform}
    Let $G$ be a LCA group. Let $k \in L^2(G)$, and let $(\varphi_{n})_{n \in \mathbb{N}}$ be a sequence in $L^2(G)$ with $\|\varphi_{n} \|_{L^{2}(G)} \le C$ for all $n \in \mathbb{N}$. If the sequence $(\varphi_{n})_{n \in \mathbb{N}}$ converges weakly in $L^2(G)$ to a function $\varphi \in L^2(G)$, then the sequence $(k \ast \varphi_{n})_{n \in \mathbb{N}}$  converges to $k \ast \varphi$ uniformly on the compact subsets of $G$. 
\end{lemma}

\begin{proof}
    For $h \in L^2(G)$, define 

    \begin{equation*}
        L(h):= \int_{G}h(g)\varphi(-g)\mbox{d}\lambda_{G}(g), \hspace{0.1cm} L_{n}(h):= \int_{G}h(g)\varphi_{n}(-g)\mbox{d}\lambda_{G}(g).
    \end{equation*}

The functionals $L$ and $L_{n}$ are continuous with respect to the norm topology on $L^2(G)$. By H\"{o}lder's inequality, $|L_{n}(h)| \le C\|h\|_{L^2(G)}$ for all $n \in \mathbb{N}$. Therefore, by the Uniform Boundedness Principle, the family $\{L_{n}: n \in \mathbb{N} \}$ is equicontinuous on $L^2(G)$. Note that weak convergence of $(\varphi _n)_{n \in \mathbb{N}}$ to $\varphi$ is equivalent to pointwise convergence of $L_{n}$ to $L$, and that pointwise convergence together with equicontinuity of $(L_{n})_{n \in \mathbb{N}}$ implies that $(L_{n})_{n \in \mathbb{N}}$ converges to $L$ uniformly on the compact subsets of $L^2(G)$. Let $K$ be a compact subset of $G$. The mapping $G \ni x \mapsto \leftindex _x{k} \in L^2(G)$, where $\leftindex_x{k}$ denotes the left translate of $k$ by an element $x \in G$ and given by $_x k(g) = k(x+g)$ for all $g \in G$, is continuous and hence the set $\{\leftindex_x{k}: x \in K\}$ is compact in the norm topology of $L^2(G)$. The uniform convergence on $K$ of $(k \ast \varphi_{n})_{n \in \mathbb{N}}$ now follows from the equations

\begin{equation*}
    k \ast \varphi_{n}(x) = \int_{G}k(x+g)\varphi_{n}(-g)\mbox{d}\lambda_{G}(g) = L_{n}(\leftindex_{x}k) \mbox{ and } k \ast \varphi(x) = L(\leftindex_x{k}).
\end{equation*}
\end{proof}

\begin{theorem}[Compare \cite{sasvari}, Theorem 1.5.2]\label{uniformconvergence}
    Let $G$ be a LCA group. If a sequence $(f_{n})_{n \in \mathbb{N}}$ of continuous positive definite functions on $G$ converges weakly in $L^2(G)$ to a continuous positive definite function $f$ and $\lim_{n \to \infty} f_{n}(0) = f(0)$, then this sequence converges to $f$ uniformly on the compact subsets of $G$.
\end{theorem}

\begin{proof}
    Let $K$ be a fixed compact subset of $G$. For any function $h$ on $G$, write $\|h\|_{K} := \sup\{|h(g)| : g \in K\}$ and $\|h\|_{\infty} :=\ \sup \{|h(g)|: g \in G\}$. We show that 

    \begin{equation*}
        \lim_{n \to \infty}\| f_{n} -f \|_K = 0.
    \end{equation*}

Let $V$ and $\chi$ be as in Lemma \ref{integralinequality}. We have

\begin{equation*}
    \|f_n - f\|_K \le \|f_n - f_n \ast \chi\|_{\infty} + \|f_n \ast \chi - f \ast \chi \|_K + \|f \ast \chi - f \|_{\infty}.
\end{equation*}

Now let $\varepsilon > 0$ be arbitrary. By the continuity of $f$, we can choose $V$ small enough so that

\begin{equation*}
    2f(0)\lambda_{G}(V)^{-1}\int_{V}(f(0) - \mathrm{Re}f(g))\mbox{d}\lambda_{G}(g) < \frac{\varepsilon ^2}{9}.
\end{equation*}

From Lemma \ref{integralinequality}, we have

\begin{equation*}
    \|f \ast \chi - f\|_{\infty} < \frac{\varepsilon}{3}.
\end{equation*}

Since $(f_{n})_{n \in \mathbb{N}}$ converges to $f$ weakly in $L^2(G)$ and $\lim _{n \to \infty}f_{n}(0) = f(0)$, there exists $N >0$ such that for all $n \in \mathbb{N}$ with $n > N$ we have

\begin{equation*}
    2f_{n}(0)\lambda_{G}(V)^{-1}\int_{V}(f_{n}(0)-\mathrm{Re}f_{n}(g))\mbox{d}\lambda_{G}(g) < \frac{\varepsilon ^2}{9}.
\end{equation*}

Applying Lemma \ref{integralinequality} again, we get 

\begin{equation*}
    \|f_n - f_n \ast \chi \|_{\infty} < \frac{\varepsilon}{3}.
\end{equation*}

By Lemma \ref{weakuniform}, there exists $M > N$ such that for all $n \in \mathbb{N}$ with $n > M$ we have

\begin{equation*}
    \|f_n \ast \chi - f \ast \chi \|_K < \frac{\varepsilon}{3}.
\end{equation*}

Altogether, we conclude that $\| f_n - f \|_{K} < \varepsilon$ for all $n > M$, showing that $(f_n)_{n \in \mathbb{N}}$ converges uniformly on the compact subset $K$.

\end{proof}

We shall a version of Mazur's Lemma. Note that for a subset $A$ of a Banach space $E$, $\mathrm{conv}(A)$ refers to the convex hull of $A$. 

\begin{lemma}[Mazur's lemma, see \cite{brezis}, Corollary 3.8 and Exercise 3.4] \label{mazur} Let $E$ be a Banach space and let $(x_{n})_{n \in \mathbb{N}}$ be a sequence in $E$ converging to $x \in E$ weakly, then there exists a sequence $(y_{n})_{n \in \mathbb{N}}$ in $E$ such that 

\begin{equation*}
    y_{n} \in \mathrm{conv}\left(\bigcup_{i=n}^{\infty}\{x_{i}\}\right) \mbox{ for all } n \in \mathbb{N},
\end{equation*}

and $(y_n)_{n \in \mathbb{N}}$ converges strongly to $x$.

\end{lemma}

\begin{theorem}\label{sigmacompactcase}
    Let $G$ be a $\sigma$-compact LCA group and $\Omega$ a closed neighbourhood of $0$ with finite Haar measure. There exists an extremal function for $\mathcal{D}_{G}(\Omega)$. Moreover, the extremal function can be approximated uniformly on the compact subsets of $G$ by a $\mathcal{D}_{G}(\Omega)$-extremal sequence.
\end{theorem}

\begin{proof}

Let $(f_{n})_{n \in \mathbb{N}}$ be a $\mathcal{D}_{G}(\Omega)$-extremal sequence in $\mathcal{G}_{G}(\Omega)$ chosen such that

\begin{equation*}
    \int_{G}f_{n}(g)\mbox{d}\lambda_{G}(g) \ge \mathcal{D}_{G}(\Omega) - \frac{1}{n} \mbox{ for all } n \in \mathbb{N}.
\end{equation*}

Since $\mathcal{G}_{G}(\Omega)$ is contained in a closed, bounded, and hence a weakly sequentially compact subset of $L^2(G)$, the sequence $(f_{n})_{n \in \mathbb{N}}$ has a subsequence converging weakly in $L^2(G)$ to some $f \in L^2(G)$. By replacing $(f_n)_{n \in \mathbb{N}}$ by this subsequence, we assume that $(f_n)_{n \in \mathbb{N}}$ converges weakly in $L^2(G)$ to $f$. Since $\mathcal{G}_{G}(\Omega)$ is a convex subset of $L^2(G)$, Lemma \ref{mazur} implies that there is a sequence $(h_{n})_{n \in \mathbb{N}}$ in $\mathcal{G}_{G}(\Omega)$ chosen such that $h_{n} \in \mathrm{conv} (\bigcup _{i =n}^{\infty}{f_{i}})$ for all $n \in \mathbb{N}$, and converging to $f$ strongly in $L^2(G)$. We claim that the latter sequence is $\mathcal{D}_{G}(\Omega)$-extremal. First note that by the extremality of $(f_{n})_{n \in \mathbb{N}}$, we have that for any $\varepsilon > 0$ there exists $N \in \mathbb{N}$ such that for all $n > N$ it holds that 

\begin{equation*}
    \left |\int_{G}f_{n}(g)\mbox{d}\lambda_{G}(g) - \mathcal{D}_{G}(\Omega) \right| < \varepsilon.
\end{equation*}

Now, write $h_{n} = \sum_{k =n}^{\infty}c_{k}^{n}f_{k}$ where $c_{k}^{n} \ge 0$, $\sum_{k = n}^{\infty}c_{k}^{n} = 1$, and $c_{k}^{n} = 0$ for all but finitely many $k \ge n$. Then for any $n > N$, we have 

\begin{equation*}
       \left | \int_{G}h_{n}(g)\mbox{d}\lambda_{G}(g)  - \mathcal{D}_{G}(\Omega) \right| \le \sum_{k =n}^{\infty}c_{k}^{n}\left|\int_{G}f_{k}(g)\mbox{d}\lambda_{G}(g) - \mathcal{D}_{G}(\Omega) \right| < \sum_{k = n}c_{k}^{n}\varepsilon = \varepsilon.
\end{equation*}

Therefore $(h_{n})_{n \in \mathbb{N}}$ is $\mathcal{D}_{G}(\Omega)$-extremal. So, replacing, if necessary, the sequence $(f_{n})_{n \in \mathbb{N}}$ by this new sequence $(h_{n})_{n \in \mathbb{N}}$, assume that $(f_{n})_{n \in \mathbb{N}}$ converges to $f$ strongly in $L^2(G)$. By passing to a pointwise almost everywhere convergent subsequence, assume that $(f_{n})_{n \in \mathbb{N}}$ converges to $f$ pointwise almost everywhere. The fact that $|f_{n}(g)| \le 1$ for all $n \in \mathbb{N}$ $g \in G$, and that $(f_{n})_{n \in \mathbb{N}}$ converges almost everywhere to $f$ implies that $f$ is bounded in $L^{\infty}(G)$ by 1. Since $(f_{n})_{n \in \mathbb{N}}$ converges weakly in $L^2(G)$ to $f$, it follows that $f$ is integrally positive definite. The function $f$ being integrally positive definite and $G$ being $\sigma$-compact implies that $f$ agrees almost everywhere with a continuous positive definite function. Thus, by correcting $f$ on a set of Haar measure zero, thus not changing the value of its integral, assume that $f$ is a continuous positive definite function. Continuity of $f$ and the fact that $f$ is a pointwise almost everywhere limit of the functions $f_{n}$, each of which satisfying $\supp (f_{n})_{+} \subset \Omega$, implies that $\supp f_{+} \subset \Omega$. Since $\supp f_{+} \subset \Omega$ we have 

       \begin{equation}\label{positivitybound}
        \int_{G}f_{+}(g)\textup{d}\lambda_{G}(g) = \int_{\Omega}f_{+}(g)\textup{d}\lambda_{G}(g) \le \lambda_{G}(\Omega) < \infty.
    \end{equation}
    On the other hand, by Fatou's lemma, 

    \begin{equation}\label{negativitybound}
        \begin{split}
            \int_{G}f_{-}(g)\textup{d}\lambda_{G}(g) & \le \liminf_{n \to \infty} \int_{G}(f_{n})_{-}(g)\textup{d}\lambda_{G}(g)\\ 
            &\le \liminf_{n \to \infty}\int_{G}(f_{n})_{+}(g)\textup{d}\lambda_{G}(g) \\
            & \le \lambda_{G}(\Omega) < \infty,    
        \end{split}
    \end{equation}

    where the second inequality of \eqref{negativitybound} is obtained using that $f_{n}$ being positive definite and integrable implies that
    \begin{equation*}
       0 \le \int_{G}f_{n}(g)\mbox{d}\lambda_{G}(g) = \int_{G}(f_{n})_{+}(g)\mbox{d}\lambda_{G}(g) - \int_{G}(f_{n})_{-}(g)\mbox{d}\lambda_{G}(g).
    \end{equation*} 

     Together, \eqref{positivitybound} and \eqref{negativitybound} imply that 
\begin{equation*}
      \int_{G}|f(g)|\textup{d}\lambda_{G}(g) =   \int_{G}f_{+}(g)\textup{d}\lambda_{G}(g) +  \int_{G}f_{-}(g)\textup{d}\lambda_{G}(g) < \infty.
\end{equation*}

Therefore, $f \in L^1(G)$. Next, we show that $f \in \mathcal{G}_{G}(\Omega)$ and that it is extremal. To that end, note that the sequence $((f_{n})_{+})_{n \in \mathbb{N}}$ of the positive parts of the $f_{n}$'s converges pointwise almost everywhere to $f_{+}$, and that $(f_{n})_{+} \le \mathbf{1}_{\Omega}$. Therefore, by Lebesgue's dominated convergence theorem and $\lambda_{G}(\Omega) < \infty$, we have

  \begin{equation}\label{eq10}
     \int_{G}f_{+}(g)\mbox{d}\lambda_{G}(g) =  \lim_{n \to \infty} \int_{G}(f_{n})_{+}(g)\mbox{d}\lambda_{G}(g).
  \end{equation}

  Recall that from  \eqref{negativitybound} we have the inequality

  \begin{equation}\label{eq9}
            \int_{G}f_{-}(g)\lambda_{G}(g) \le \liminf_{n \to \infty}\int_{G}(f_{n})_{-}(g)\mbox{d}\lambda_{G}(g).
 \end{equation}
  
  Now, subtracting \eqref{eq9} from \eqref{eq10} we have

  \begin{equation*}
\begin{split}
    \int_G f(g)  \text{d}\lambda_G(g) &\ge  \lim_{n \to \infty} \int_{G}(f_n)_{+}(g)\text{d}\lambda_G(g) - \liminf_{n \to \infty}{\int_{G}(f_n)_{-}(g)} \text{d}\lambda_G(g) \\
    &= \limsup_{n \to \infty}{\int_{G}\left((f_n)_{+}(g) - (f_n)_{-}(g)\right) \text{d}\lambda_G(g)} \\
    &= \lim_{n \to \infty}\int_{G}f_n (g) \mbox{d}\lambda _{G}(g) = \mathcal{D}_{G}(\Omega).
\end{split}
\end{equation*}
		
	

Since $f$ is continuous, positive definite, integrable, and $\supp f_{+} \subset \Omega$ have already been established, to have that $f \in \mathcal{G}_{G}(\Omega)$ we only need to show that $f(0)=1$. Since 

\begin{equation*}
    \int_{G}f(g)\mbox{d}\lambda_{G}(\Omega) \ge \mathcal{D}_{G}(\Omega) > 0,
\end{equation*}

$f$ is not the constantly zero function. Therefore, $f(0) \ne 0$. If $0 < f(0) <1$, then the function $h:= f/f(0)$ is in $\mathcal{G}_{G}(\Omega)$ and has

 \begin{equation*}
     \int_{G}h(g)\mbox{d}\lambda_{G}(g) = \frac{1}{f(0)}\int_{G}f(g)\mbox{d}\lambda_{G}(\Omega) \ge \frac{\mathcal{D}_{G}(\Omega)}{f(0)} > \mathcal{D}_{G}(\Omega),
 \end{equation*}

contradicting the extremality of $\mathcal{D}_{G}(\Omega)$. Therefore, $f(0) = 1$ and hence $f \in \mathcal{G}_{G}(\Omega)$. Note that this also shows that $f$ is extremal. Finally, since $(f_{n})_{n \in \mathbb{N}}$ converges to $f$ weakly in $L^2(G)$ and $\lim_{n \to \infty} f_{n}(0) = f(0)$, Theorem \ref{uniformconvergence} implies that $(f_{n})_{n \in \mathbb{N}}$ converges to $f$ uniformly on the compact subsets of $G$. 

\end{proof}

Now, having secured the existence of an extremal function for $\sigma$-compact groups, we extend the solution to general LCA groups to prove our main result.

\begin{theorem}\label{generalcase}
     Let $G$ be a LCA group and $\Omega$ a closed neighbourhood of $0$ with finite Haar measure. There exists an extremal function for $\mathcal{D}_{G}(\Omega)$. Moreover, the extremal function can be approximated uniformly on the compact subsets of $G$ by a $\mathcal{D}_{G}(\Omega)$-extremal sequence.
\end{theorem}

\begin{proof}
    On account of Theorem \ref{reduction}, let $H$ be an open $\sigma$-compact subgroup of $G$ containing $\Omega$ such that $\mathcal{D}_{G}(\Omega) = \mathcal{D}_{H}(\Omega)$. By Theorem \ref{sigmacompactcase}, there exists an extremal function $f \in \mathcal{G}_{H}(\Omega)$ such that 

    \begin{equation*}
        \int_{H}f(g)\mbox{d}\lambda_{H}(g) = \mathcal{D}_{H}(\Omega).
    \end{equation*}

    Let $\widetilde{f}$ be the trivial extension of $f$. By Lemma \ref{extensionlemma}, $\widetilde{f} \in \mathcal{G}_{G}(\Omega)$. We claim that $\widetilde{f}$ is an extremal function for $\mathcal{D}_{G}(\Omega)$. Indeed, 
			
	\begin{equation*}
		\int_{G}\widetilde{f}(g)\textup{d}\lambda_{G}(g) = \int_{H}f(g)\textup{d}\lambda_{H}(g) =             \mathcal{D}_{H}(\Omega) = \mathcal{D}_{G}(\Omega).
	\end{equation*}

 As in Theorem \ref{sigmacompactcase}, $(f_n)_{n \in \mathbb{N}}$ converges to $f$ uniformly on the compact subsets of $H$, and hence $(\widetilde{f}_{n})_{n \in \mathbb{N}}$ converges to $\widetilde{f}$ uniformly on the compact subsets of $G$.

\end{proof}

\section*{Acknowledgements}

The author would like to sincerely thank Elena E. Berdysheva and Szil\'{a}rd Gy. ~ R\'{e}v\'{e}sz for their invaluable supervision and encouragement throughout this project. The author is also thankful to the anonymous reviewer for the feedback that helped improve the paper. Finally, the author would also like to thank the Shuttleworth Postgraduate Scholarship of the University of Cape Town for the financial support that made this research possible.


\begin{thebibliography}{99}


        \bibitem{arbab1} V. V. Arestov and A. G. Babenko, \textit{On the Delsarte scheme for estimating contact numbers} (Russian), Tr. Mat. Inst. Steklova \textbf{219} (1997), 44–73; translation in Proc. Steklov Inst. Math. (4) \textbf{219} (1997), 36–65.

        \bibitem{arbab2} V. V. Arestov and A. G. Babenko, \textit{Estimates of the maximal value of angular code distance for 24 and 25 points on the unit sphere in $\mathbb{R}^4$}, Math. Notes (4) \textbf{68} (2000), 419–435.
 

        \bibitem{elena-szilard} E. Berdysheva, Sz. Gy. R\'{e}v\'{e}sz, \textit{Delsarte’s extremal problem and packing on locally compact abelian groups}, Annali della Scuola Normale di Pisa – Classe di Scienze (5) \textbf{XXIV} (2023), 1007–1052.

        \bibitem{pbab} P. Boyvalenkov, S. Dodunekov, and O. Musin, \textit{A survey on the kissing numbers}, Serdica Math. J. (4) \textbf{38} (2012), 507–522.

        \bibitem{cohn-elkies}H. Cohn and N. Elkies, \textit{New upper bounds for sphere packings, I}, Ann. of Math. (2) \textbf{157} (2003), 689–714.

        \bibitem{deitmar} A. Deitmar, S. Echterhoff, \say{Principles of Harmonic Analysis}, Second Edition, Springer (2014).

        \bibitem{delsarte1} P. Delsarte, \textit{Bounds for unrestricted codes by linear programming}, Philips Res. Rep. \textbf{2} (1972), 272–289.


        \bibitem{delsarte2} P. Delsarte, J.-M. Goethals, and J. J. Seidel, \textit{Spherical codes and designs, Geom. Dedicata} (3) \textbf{6} (1977), 363–388.

      
        \bibitem{brezis} H. Brezis, \say{Functional Analysis, Sobolev Spaces, and Partial Differential Equations}, Springer-Verlag, New York (2010).
        
        \bibitem{marcell-zsuzsa} M. Ga\'{a}l and Zs. Nagy-Csiha, \textit{On the  existence of an extremal function
			in the Delsarte extremal problem}, Mediterr. J. Math. (2020), \textbf{17}:190.


        \bibitem{gorbachev} D. V. Gorbachev, \textit{Extremal problem for entire functions of exponential spherical type, connected with the Levenshtein bound on the sphere packing density in $\mathbb{R}^n$} (Russian), Izv. Tula State Univ. Ser. Mat. Mekh. Inform. \textbf{6} (2000), 71–78.

        \bibitem{gorbachev1} D. V. Gorbachev, \textit{An extremal problem for periodic functions with supports in the ball}, Math. Notes (3) \textbf{69} (2001), 313–319.

        \bibitem{hewittross} E. Hewitt, K. Ross, \say{Abstract Harmonic Analysis}, II, Die Grundlehren der mathematischen Wissenchaften, vol. 152. Springer, Berlin, Heidelberg, New York, Budapest (1970).

        \bibitem{koblev} G.A. Kabatyanskii and V.I. Levenshtein, \textit{On bounds for packing on the sphere and in space} (Russian), Probl. Inform. (1) \textbf{14} (1978), 3–25.

        \bibitem{kolrev} M. N. Kolountzakis and Sz. Gy. Révész, \textit{Tur\'{a}n’s extremal problem for positive definite functions on groups}, J. London Math. Soc. \textbf{74} (2006), 475–496.

        \bibitem{kuk} N.A. Kuklin, \textit{Delsarte method in the problem on kissing numbers in high-dimensional spaces}, Proc. Steklov Inst. Math. Suppl. 1 \textbf{284} (2014), S108–S123.

        \bibitem{sasvari} Z. Sasv\'{a}ri, \say{Positive Definite and Definitizable Functions}, Vol. 2, Akademie Verlag, Berlin (1994).
        
        \bibitem{shtrom} D.V. Shtrom, \textit{The Delsarte method in the problem of the antipodal contact numbers of Euclidean spaces of high dimensions}, arXiv:math/0306405v1.
        	

        \bibitem{viazovska} M. S. Viazovska, \textit{The sphere packing problem in dimension 8}, Ann. of Math. (3) \textbf{185} (2017), 991–1015.
    
  \end{thebibliography}
\end{document}